\documentclass[pre,preprint,nofootinbib,floatfix]{revtex4-1}

\usepackage[colorlinks,hyperfootnotes=false]{hyperref}
\hypersetup{allcolors=black}
\usepackage{amsmath}
\usepackage{amssymb}
\usepackage{amsthm}
\usepackage{amsfonts}
\usepackage{amscd}
\usepackage{bm}
\usepackage{bbm}
\usepackage{mathtools}

\newtheorem{theorem}{Theorem}
\newtheorem{definition}[theorem]{Definition}
\newtheorem{assumption}[theorem]{Assumption}

\DeclareMathOperator*{\argmin}{arg\,min}

\renewcommand{\Re}{\mathbb{R}}

\renewcommand{\citet}[1]{\cite{#1}}

\begin{document}

\title{ADMM and Accelerated ADMM as \\ Continuous Dynamical Systems}

\author{Guilherme Fran\c ca}
\email{guifranca@gmail.com}

\author{Daniel Robinson}
\email{daniel.p.robinson@jhu.edu}

\author{Ren\' e Vidal}
\email{rvidal@jhu.edu}

\affiliation{Mathematical Institute for Data Science, 
Johns Hopkins University, Baltimore MD 21218, USA}

\begin{abstract}
Recently, there has been an increasing interest in using tools from dynamical
systems to analyze the behavior of simple optimization algorithms such as
gradient descent and accelerated variants. This paper strengthens such
connections by deriving the differential equations that model the continuous
limit of the sequence of iterates generated by the alternating direction method
of multipliers, as well as an accelerated variant. We employ the direct method
of Lyapunov to analyze the stability of critical points of the dynamical
systems and to obtain associated convergence rates.
\end{abstract}

\maketitle

\section{Introduction}

The theory of dynamical systems has been extensively developed since its
origins by Poincar\' e in the late 19th century. For example, the work of
Lyapunov on stability is commonly used in physics,
control systems, and other branches of applied mathematics. However, 
the connection between dynamical systems and optimization algorithms has only
recently been studied. The basic idea is that tools from dynamical
systems can be used to analyze the stability and convergence rates of the
continuous 
limit of the sequence of iterates generated by optimization algorithms.
Prior work has established  
these connections for simple optimization algorithms
such as gradient descent (GD) and accelerated gradient descent (A-GD). This
paper 
improves upon prior work
by deriving the dynamical systems associated 
with the continuous limit of two 
commonly used optimization methods: the alternating
direction method of multipliers (ADMM) and an accelerated version of ADMM
(A-ADMM). Moreover, this paper analyzes the stability properties of the
resulting dynamical systems and derives their convergence rates, 
which for ADMM matches the known rate of its discrete 
counterpart, and for A-ADMM provides a new result.

\subsection{Related work}
Perhaps the simplest connection between an optimization algorithm
and a continuous dynamical system is exhibited by the GD method. 
The GD algorithm aims to minimize a function $f:\mathbb{R}^n \to \mathbb{R}$ 
via the update
\begin{equation}
\label{eq:gd}
x_{k+1} = x_k - \eta \nabla f(x_k) ,
\end{equation}
where $x_k$ denotes the $k$th solution 
estimate and $\eta > 0$ is the step size. 
It is immediate that a continuous limit of the iterate update~\eqref{eq:gd} 
leads to the gradient flow
\begin{equation}
\label{eq:gd_ode}
\dot{X} = -\nabla f(X) ,
\end{equation}
where $X = X(t)$ is the continuous limit of  $x_k$ and
$\dot{X} \equiv \tfrac{dX}{dt}$ denotes its time
derivative. It is known that GD has a convergence rate of $O(1/k)$ 
for convex functions~\cite{Nesterov:2004}.
Interestingly, the differential equation \eqref{eq:gd_ode} 
also has a convergence rate of $O(1/t)$, which is consistent with GD.

Nesterov \citet{Nesterov:1983} proposed an accelerated GD algorithm, 
henceforth referred to as  A-GD, by adding momentum to the $x$ variables.  
The update for A-GD may be written as
\begin{subequations}
\label{eq:agd}
\begin{align}
x_{k+1} &= \hat{x}_{k} - \eta \nabla f(\hat{x}_k),  \\
\hat{x}_{k+1} &= x_k + \dfrac{k}{k+r}(x_{k+1}-x_k),
\end{align}
\end{subequations}
where $r \ge 3$ ($r=3$ is the standard choice) 
and $\hat{x}_k$ denotes the $k$th accelerated vector. 
It is known that A-GD has a convergence rate of $O(1/k^2)$ for 
convex functions, which is known to be optimal in 
the sense of worst-case complexity~\cite{Nesterov:2004}.
Recently, the continuous limit
of A-GD was computed by \citet{Candes:2016} to be
\begin{equation}
\label{eq:nesterov_ode}
\ddot{X} + \dfrac{r}{t} \dot{X} + \nabla f(X) = 0,
\end{equation}
where $\ddot{X} \equiv \tfrac{d^2 X}{dt^2}$ is the acceleration.
By using a Lyapunov function, it was shown
that the convergence rate associated with 
\eqref{eq:nesterov_ode} is $O(1/t^2)$ when $f$ 
is convex \cite{Candes:2016}, 
which matches the known rate of $O(1/k^2)$ for A-GD.

To better
understand the acceleration
mechanism,
a variational approach 
was proposed~\cite{Wibisono:2016} for which 
the continuous limit of a class of accelerated 
methods was obtained using the Bregman Lagrangian; 
the class of methods includes A-GD,  
its non-Euclidean extension, and accelerated
higher-order gradient methods. Also, 
a differential equation modeling the continuous limit
of accelerated mirror descent was 
obtained \cite{Krichene:2015}.

While \citet{Candes:2016} focused on the discrete to continuous limit,
\citet{Wibisono:2016,Krichene:2015} stress the converse, by
which one starts with a second-order differential equation and then
constructs a discretization 
with a matching convergence rate. This approach can lead to new
accelerated algorithms. Indeed,
\citet{Krichene:2015} introduced a family of 
first-order accelerated methods and established their 
convergence rates by using a 
discrete Lyapunov function, which is analogous to its
continuous counterpart. 
In related work, \citet{Wilson:2016} proposed 
continuous and discrete time Lyapunov frameworks
for \mbox{A-GD} based methods that built additional connections 
between rates of convergence and the choice of discretization. 
In particular, they showed that a
naive discretization can produce iterates that 
do not match the convergence rate of the differential
equation and proposed 
rate-matching algorithms \cite{Wibisono:2016,Krichene:2015}. However, 
such rate-matching  algorithms 
introduce extra conditions such as an
intermediate sequence or a 
new function that must 
obey certain constraints, which is in stark contrast to A-GD.

An important algorithm commonly used in machine learning and statistics 
is ADMM
\cite{Gabay:1976,Glowinsky:1975,Boyd:2011,Eckstein:2015},
which is 
often more easily distributed when compared to its 
competitors and hence appealing for 
large scale applications. There are some interesting 
relations between
ADMM and \emph{discrete} dynamical systems. For instance,
the formalism of integral
quadratic constraints \cite{Lessard:2016} was applied to ADMM
\cite{Nishihara:2015}
under the assumption of strong convexity.
Based on this, \citet{FrancaBento:2016} establish an explicit upper
bound on the convergence rate of ADMM in terms of the algorithm 
parameters and the condition number of the problem by analytically solving
the semi-definite program introduced by \citet{Nishihara:2015}.
Moreover, 
for a class of convex quadratic consensus problems defined over a graph,
ADMM can be viewed as a lifted Markov chain
\cite{FrancaBento:2017,FrancaBento:2017_2} that exhibits 
a significant speedup in the mixing time compared to GD, 
which corresponds to the base Markov chain.
For convex functions, 
ADMM has an $O(1/k)$ convergence rate \cite{Eckstein:2015}. Recently,  
using the ideas of \citet{Nesterov:1983}, \citet{Goldstein:2014} proposed
an accelerated version of ADMM (henceforth 
called \mbox{ A-ADMM}) and established
a convergence rate of $O(1/k^2)$ when both $f$ and $g$ are \emph{strongly} 
convex functions, and moreover $g$ is a  quadratic function.

Although
there is extensive literature on 
ADMM and its variants, connections between their 
continuous limit and differential equations is unknown.
This paper is a first step in establishing such connections in the context 
of the problem\footnote{
The standard problem $\min_{x,z} f(x) + g(z)$ 
subject to $A x + Bz = c$ can be 
recovered by redefining $A$ when $B$ is invertible.
}
\begin{equation}
\label{eq:minimize}
\min_{x, z} \ \{ V(x,z) = f(x) + g(z)\} \qquad 
\mbox{subject to} \ \  z = A x,
\end{equation}
where $f:\Re^n\to\Re$ and $g:\Re^m \to \Re$ are  
continuously differentiable convex 
functions, 
$x\in\Re^n$, $z\in\Re^m$, $A \in\mathbb{R}^{m\times n}$ and $m \geq n$.
The problem formulation~\eqref{eq:minimize} covers many 
interesting applications 
in machine learning and statistics.
With that said, we also recognize that many important problems 
do not fall within the framework~\eqref{eq:minimize} due
to the assumption of differentiability, especially of $g$ which is usually
a regularization term. Such an assumption is a theoretical necessity 
to allow connections to differential equations\footnote{The 
differentiability assumption can be relaxed by using subdifferential calculus
and differential inclusions, but this is beyond the scope of this work.
}.

\subsection{Paper contributions}
Our first contribution is to show 
in Theorem~\ref{th:admm_ode} that the dynamical system 
that is the continuous limit of ADMM when applied to~\eqref{eq:minimize}
is given by the ADMM flow
\begin{equation}
\label{eq:admm1diff}
\big(A^T A\big) \dot X +  \nabla V(X) = 0.
\end{equation}
Note that when $A = I$ we obtain the dynamical 
system \eqref{eq:gd} (i.e., the continuous limit of GD), which can be 
thought of as an unconstrained formulation of~\eqref{eq:minimize}.  
Our second contribution is to show in 
Theorem~\ref{th:fast_admm_ode} that the 
dynamical system 
that is the continuous 
limit of A-ADMM
is given by the A-ADMM flow
\begin{equation}
\label{eq:fadmm1diff}
\big( A^T A \big) \left( \ddot{X} + 
\dfrac{r}{t} \dot{X}\right) + \nabla V(X) = 0. 
\end{equation}
Here, the dynamical 
system~\eqref{eq:nesterov_ode} that is the continuous 
limit of A-GD is a particular case obtained when $A = I$.
We then employ the direct method of Lyapunov to study 
the stability properties of 
both dynamical systems \eqref{eq:admm1diff} and \eqref{eq:fadmm1diff}.
We show that under reasonable assumptions on $f$ and $g$,
these dynamical systems are 
asymptotically stable.  
Also, we prove that \eqref{eq:admm1diff}  
has a convergence rate of $O(1/t)$, whereas \eqref{eq:fadmm1diff} has
a convergence rate of $O(1/t^2)$.

\subsection{Notation}
We use $\|\cdot\|$ 
to denote the Euclidean two norm, and $\langle u, v \rangle = u^T v $
to denote the inner product of $u, v\in\Re^n$. 
For our analysis, it is convenient to define the function
\begin{equation} \label{def.Vx}
V(x) = f(x) + g(Ax),
\end{equation}
which is closely related to the function $V(x,z)$ that 
defines the objective function in~\eqref{eq:minimize}. 
In particular, for all $(x,z)$ satisfying $z = Ax$, 
the relationship $V(x,z) = V(x)$ holds.
Throughout the paper, we make the following assumption.
\begin{assumption}\label{ass.basic}
The functions 
$f$ and $g$ in \eqref{eq:minimize} 
are continuously differentiable and convex, 
and $A$ has full column~rank.
\end{assumption}

\section{Continuous Dynamical Systems}
\label{sec:admm_ode}

In this section, we show that the continuous 
limits of the ADMM and A-ADMM algorithms are first- and second-order 
differential equations, respectively.

\subsection{ADMM}
\label{sec:admm}

The scaled form of ADMM
is given by
\cite{Boyd:2011} 
\begin{subequations}
\label{eq:admm}
\begin{align}
x_{k+1} &= 
\argmin_{x\in\Re^n} \ f(x) + \tfrac{\rho}{2} \| A x - z_k + u_k \|^2,
\label{eq:admm1} \\
z_{k+1} &= \argmin_{z\in\Re^m} \
g(z) + \tfrac{\rho}{2} \| A x_{k+1} - z + u_k \|^2 , 
\label{eq:admm2} \\
u_{k+1} &= u_k + Ax_{k+1} - z_{k+1},
\label{eq:admm3}
\end{align}
\end{subequations}
where $\rho > 0$ is a penalty parameter and $u_k\in\Re^m$ is the $k$th 
Lagrange multiplier estimate for the constraint $z = Ax$.  
Our next result shows how a continuous limit of the ADMM updates 
leads to a particular first-order differential equation.

\begin{theorem}
\label{th:admm_ode}
Consider the optimization problem \eqref{eq:minimize} and the 
associated function $V(\cdot)$ in~\eqref{def.Vx}.
The continuous limit associated with the ADMM updates in \eqref{eq:admm},
with time scale $t = k/\rho$,
corresponds to the initial value problem
\begin{equation}
\label{eq:admm_ode}
\dot X + \big(A^T A\big)^{-1} \nabla V(X) = 0 
%\quad \mbox{with $X(0) = x_0$.}
\end{equation}
with $X(0) = x_0$.
\end{theorem}
\begin{proof}
Since $f$ and $g$ are convex and $A$ has full column rank 
(see Assumption~\ref{ass.basic}), the optimization 
problems in \eqref{eq:admm1} 
and \eqref{eq:admm2} are strongly convex so that $(x_{k+1},z_{k+1},u_{k+1})$ 
is unique.  It follows from the optimality conditions 
for problems~\eqref{eq:admm1} and~\eqref{eq:admm2} 
that $(x_{k+1},z_{k+1},u_{k+1})$ satisfies
\begin{subequations}
\begin{align}
\nabla f\left(x_{k+1}\right) + \rho A^T\left(Ax_{k+1} - z_k + u_k \right) 
&= 0, \\
\nabla g\left(z_{k+1}\right) - \rho \left(Ax_{k+1} - z_{k+1} 
+ u_k \right) &= 0, \\
u_{k+1} -  u_k - Ax_{k+1} + z_{k+1} &= 0,
\label{eq:admm3_2}
\end{align}
\end{subequations}
which can be combined to obtain
\begin{equation}
\label{eq:admm_combined}
\nabla f\left(x_{k+1}\right) + A^T \nabla g\left(z_{k+1}\right) + 
\rho A^T\left( z_{k+1} - z_k \right) = 0 .
\end{equation}

Let  $t = \delta k$ and
$x_k = X(t)$, 
with a similar notation 
for $z_k$ and $u_k$. Using the Mean Value Theorem on the $i$th 
component of $z_{k+1}$
we have that $z_{k+1,i} = Z_i(t+\delta) = Z_i(t) + \delta
\dot{Z}_i(t+\lambda_i\delta)$ for some $\lambda_i\in[0,1]$. Therefore,
\begin{equation}
\label{eq:first_proof_der}
\lim_{\delta\to 0}
\dfrac{z_{k+1, i} - z_{k,i}}{\delta} = \lim_{\delta\to 0}
\dot{Z}_i(t+\lambda_i\delta) 
= \dot{Z}_i(t).
\end{equation}
Since this holds for every component $i=1,\dotsc,m$ we see that, in the
limit $\delta \to 0$, the third 
term in \eqref{eq:admm_combined} is exactly equal to the vector $\dot{Z}(t)$,
provided we choose $\rho = 1/\delta$. For the first two terms of
\eqref{eq:admm_combined},
note that
\begin{equation}
\label{eq:first_proof_grad}
\begin{split}
\nabla f(x_{k+1})  +  A^T\nabla g(z_{k+1})  
&=  \nabla f(X(t+\delta)) + A^T\nabla g(Z(t+\delta))  \\
& \to \nabla f(X(t)) + A^T\nabla g(Z(t))
\end{split}
\end{equation}
as $\delta \to 0$.
Thus, taking the limit $\delta \to 0$ in \eqref{eq:admm_combined} 
and substituting
\eqref{eq:first_proof_der} and \eqref{eq:first_proof_grad} yields
\begin{equation}
\label{eq:admm_combined_limit}
\nabla f(X(t)) + A^T \nabla g(Z(t)) + A^T \dot{Z}(t) = 0.
\end{equation}
Let us now consider the $i$th component of \eqref{eq:admm3_2}.
By the Mean Value Theorem there exists $\lambda_i \in [0,1]$
such that
\begin{equation}
\begin{split}
0 &= 
U_i(t+\delta) - U_i(t) - (A X)_i(t+ \delta) + Z_i(t+ \delta)  \\
& = \delta \dot{U}_i(t+\lambda_i \delta) - 
(A X)_i(t+ \delta) + Z_i(t+\delta)  \\
&\to Z_i(t) -(AX)_i(t) 
\end{split}
\end{equation}
as $\delta \to 0$.
Since this holds for every $i = 1,\dotsc,m$ we conclude that
$Z(t) = AX(t)$
and $\dot{Z}(t) = A \dot{X}(t)$. Moreover, 
recalling the definition \eqref{def.Vx}, note that
\begin{equation}
\label{eq:admm_grad_V}
\begin{split}
\nabla f(X) + A^T \nabla g(Z) 
&= \nabla f(X) + A^T \nabla g(AX) \\
&= \nabla V(X) .
\end{split}
\end{equation}
Therefore, \eqref{eq:admm_combined_limit} becomes
\begin{equation}
\nabla V(X(t)) + A^T A \dot{X}(t) = 0,
\end{equation}
which is equivalent to~\eqref{eq:admm_ode} 
since $A$ has full column rank. 

Finally, since 
\eqref{eq:admm_ode} is a first-order differential equation,
the dynamics is specified 
by the initial condition $X(0) = x_0$, 
where $x_0$ is an initial solution estimate to~\eqref{eq:minimize}.
\end{proof}

We remark that the continuous limit of 
ADMM given by  \eqref{eq:admm_ode} and 
the continuous limit of GD given by \eqref{eq:gd_ode} 
are similar---first-order gradient systems---with
the only difference being the additional $(A^T A)^{-1}$ term. 
Thus,  
in the special case $A = I$, i.e., the unconstrained 
case,  the differential 
equation \eqref{eq:admm_ode} reduces to \eqref{eq:gd_ode}.

\subsection{A-ADMM}
\label{sec:a-admm}

We now consider an accelerated version of
ADMM that was originally proposed by \citet{Goldstein:2014}, 
which follows the same idea introduced by
\citet{Nesterov:1983} to accelerate GD.
The scaled A-ADMM method for solving problem~\eqref{eq:minimize}
can be written as follows:
\begin{subequations}
\label{eq:fast_admm}
\begin{align}
x_{k+1} &= \argmin_{x\in\Re^n} \ f(x) + \tfrac{\rho}{2} \| Ax - \hat{z}_k + 
\hat{u}_k\|^2, \label{prov.x.aadmm} \\
z_{k+1} &= \argmin_{z\in\Re^{m}} \ g(z) + 
\tfrac{\rho}{2} \| Ax_{k+1} - z + \hat{u}_k\|^2, \label{prov.z.aadmm} \\
u_{k+1} &= \hat{u}_k + Ax_{k+1} - z_{k+1}, \\
\hat{u}_{k+1} &= u_{k+1} + \gamma_{k+1} \left( u_{k+1}-u_k \right), \\
\hat{z}_{k+1} &= z_{k+1} + \gamma_{k+1} \left( z_{k+1}-z_k \right),
\end{align}
\end{subequations}
where $\hat{u}$ and $\hat{z}$ are the ``accelerated variables'' and
\begin{equation}
\label{eq:beta}
\gamma_{k+1} = \dfrac{k}{k+r}
\end{equation}
with $r \geq 3$.  We remark that the particular choice $r = 3$ 
produces the same asymptotic behavior as the parameter choice
in \citet{Goldstein:2014,Nesterov:1983}.
Our next result shows how a continuous 
limit of the A-ADMM updates is a second-order differential equation.

\begin{theorem}
\label{th:fast_admm_ode}
Consider the optimization problem \eqref{eq:minimize} and the 
associated function $V(\cdot)$ in~\eqref{def.Vx}.
The continuous limit 
associated with 
the A-ADMM updates in~\eqref{eq:fast_admm},
with time scale \mbox{$t = k/\sqrt{\rho}$},
corresponds to the initial value problem
\begin{equation}
\label{eq:fast_admm_ode}
\ddot{X} + \frac{r}{t} \dot{X} + 
\big(A^T A\big)^{-1} \nabla V(X) = 0 
\end{equation}
with $X(0) = x_0$ and $\dot{X}(0) = 0$.
\end{theorem}
\begin{proof}
According to Assumption~\ref{ass.basic}, the functions
$f$ and $g$ are convex and $A$ 
has full column rank, therefore the optimization 
problems \eqref{prov.x.aadmm} and \eqref{prov.z.aadmm} are strongly convex, 
making $(x_{k+1},z_{k+1},u_{k+1},\hat u_{k+1},\hat z_{k+1})$ unique.  
It thus follows from the optimality conditions  that 
\begin{subequations}
\label{eq:fast_admm_simplified}
\begin{align}
\nabla f(x_{k+1}) + A^T \nabla g(z_{k+1})  
+\rho A^T \!\left(z_{k+1}-\hat{z}_k \right) &= 0,
\label{eq:fast_admm_simplified1}
\\
u_{k+1} - \hat{u}_k - Ax_{k+1} +z_{k+1} &= 0 ,
\label{eq:fast_admm_simplified2}
\\
\hat{u}_{k+1} - u_{k+1} - \gamma_{k+1}\left( u_{k+1} - u_k\right) &= 0 ,
\label{eq:fast_admm_simplified3}
\\
\hat{z}_{k+1} - z_{k+1} - \gamma_{k+1} \left( z_{k+1} - z_k \right) &=0.
\label{eq:fast_admm_simplified4}
\end{align}
\end{subequations}
Let $t = \delta k$, $x_k = X(t)$ and similarly 
for $z_k$, $u_k$, $\hat{z}_k$ and $\hat{u}_k$.
Consider Taylor's Theorem for the $i$th component of $z_{k\pm 1}$:
\begin{equation}
\label{eq:taylor}
\begin{split}
z_{k\pm 1,i} &= Z_i(t \pm \delta) \\
& = Z_i(t) \pm \delta \dot{Z}_i(t) + 
\tfrac{1}{2}\delta^2 \ddot{Z}_i(t \pm \lambda_i^{\pm} \delta)
\end{split}
\end{equation}
for some $\lambda_i^{\pm} \in [0,1]$. Hence, 
from \eqref{eq:fast_admm_simplified4} we have
\begin{equation}
\label{eq:zzhat} 
\begin{split}
z_{k+1, i} - \hat{z}_{k,i} 
&= z_{k+1,i} - z_{k,i} - \gamma_{k} \left( z_{k,i} - z_{k-1,i}\right)  \\
&= (1-\gamma_k)\delta \dot{Z}_i(t) + \tfrac{1}{2} \delta^2
\ddot{Z}_i(t+\lambda_i^+ \delta) 
+\tfrac{1}{2}\gamma_k
\delta^2 \ddot{Z}_i(t-\lambda_i^-\delta) .
\end{split}
\end{equation}
From the definition \eqref{eq:beta}, and $t = \delta k$, we have
\begin{equation}
\label{eq:gammak}
\begin{split}
\gamma_{k} &= \dfrac{k-1}{k+r-1}  = 1 - \dfrac{r}{k+r-1} = 
1 - \dfrac{\delta r}{t + \delta(r-1)} \\
&= 1 - \dfrac{\delta r}{t} + O(\delta^2).
\end{split}
\end{equation}
Replacing this into \eqref{eq:zzhat} we obtain
\begin{equation}
\label{eq:zzhat2}
\begin{split}
\dfrac{z_{k+1,i}-\hat{z}_{k,i}}{\delta^2} &= \dfrac{r}{t} \dot{Z}_i(t)
+ \tfrac{1}{2}\ddot{Z}_i(t+\lambda_i^+ \delta)  
 + \tfrac{1}{2}\ddot{Z}_i(t-\lambda_i^- \delta) 
+ O(\delta) \\
& \to
\dfrac{r}{t}\dot{Z}_i(t)+\ddot{Z}_i(t)
\end{split}
\end{equation}
as $\delta \to 0$.
Hence, if we choose $\rho = 1/\delta^2$, then the limit of the third term in 
\eqref{eq:fast_admm_simplified1} is equal to $\tfrac{r}{t}A\dot{Z}(t) +
A\ddot{Z}(t)$.
Recalling \eqref{eq:first_proof_grad}, the limit 
of \eqref{eq:fast_admm_simplified1} as $\delta \to 0$ is thus given by
\begin{equation}
\nabla f(X) + A^T \nabla g(Z) + A^T\left( 
\dfrac{r}{t} \dot{Z}
+\ddot{Z}\right) = 0.
\end{equation}
Next, using \eqref{eq:gammak} into 
the $i$th component of 
\eqref{eq:fast_admm_simplified3} 
we obtain
\begin{equation}
\begin{split}
0 &= \hat{u}_{k,i}-u_{k,i} - \gamma_{k}(u_{k,i} - u_{k-1,i}) \\
&= \hat{U}_i(t) - U_i(t) - \left(1-O(\delta)\right)\delta \, 
\dot{U}_i(t-\lambda_i^-\delta) \\
& \to \hat{U}_i(t) - U_i(t)
\end{split}
\end{equation}
as $\delta \to 0$.
Since this holds for every component $i=1,\dotsc,n$ it follows that
$\hat{U}(t) = U(t)$.
Substituting this into
\eqref{eq:fast_admm_simplified2} implies that
$Z(t) = A X(t) $, which in turn implies $\dot{Z}(t) = A\dot{X}(t)$ and 
$\ddot{Z}(t) = A \ddot{X}(t)$.
Using these, and also \eqref{eq:admm_grad_V}, 
we obtain from \eqref{eq:fast_admm_simplified1} the differential
equation 
\begin{equation}
\label{eq:final-aadmm}
\begin{split}
\nabla V(X) + A^T A \left( \ddot{X} + 
\dfrac{r}{t} \dot{X}\right) = 0. 
\end{split}
\end{equation}
Since $A$ has full column rank, so that $A^T A$ is invertible, we 
see that  \eqref{eq:final-aadmm} is equivalent to~\eqref{eq:fast_admm_ode}.

It remains to consider the initial conditions. The first condition is 
$X(0) = x_0$, where $x_0$ is an initial estimate of a 
solution to~\eqref{eq:minimize}.  
Next, using the Mean Value Theorem, we have 
$\dot{X}_i(t) = \dot{X}_i (0) + t \ddot{X}_i(\xi_i)$ for 
some $\xi_i \in [0, t]$ and $i=1,\dotsc,n$.  Combining this 
with~\eqref{eq:fast_admm_ode} 
yields
\begin{equation}
\dot{X}_i(t) = \dot{X}_i(0) - r \dot{X}_i(\xi_i) - 
t \big[ \big(A^T A\big)^{-1} \nabla V(X(\xi_i)) \big]_i
\end{equation}
for all $i=1,\dots,n$. 
Letting $t \to 0^+$, which also 
forces $\xi_i \to 0^+$, 
we have that
$\dot{X}_i(0) = (1-r) \dot{X}_i(0)$ for each $i=1,\dotsc,n$.  
Since $r \ne 1$ by the choice of $\gamma_k$ in~\eqref{eq:beta}, it 
follows that $\dot X(0) = 0$, as claimed.
\end{proof}

We remark that the continuous limit of 
A-ADMM given by \eqref{eq:fast_admm_ode} and 
the continuous limit of A-GD given by \eqref{eq:nesterov_ode} 
are similar---second-order 
dynamical systems---with the only difference being 
the additional $(A^T A)^{-1}$ term. Therefore,  
in the special case $A = I$, i.e., the unconstrained case, 
\eqref{eq:fast_admm_ode} reduces to \eqref{eq:nesterov_ode}.

We close this section by noting one interesting difference between 
the derivations for the dynamical systems associated with 
ADMM and A-ADMM. Namely, the derivation for ADMM 
required choosing $\delta = 1/\rho$, whereas the derivation 
for A-ADMM made the choice $\delta = 1/\sqrt{\rho}$. 
Since the relationship $t = \delta k$ holds, we see that 
for fixed $k$ and $\rho > 1$, the time elapsed for A-ADMM is 
larger than that for ADMM, which highlights the 
acceleration achieved by A-ADMM.

\section{A Review of Lyapunov Stability}
\label{sec.L}

In the next section, we will use a Lyapunov stability approach to analyze the 
dynamical systems established in the previous section for 
ADMM and A-ADMM, namely~\eqref{eq:admm_ode} 
and~\eqref{eq:fast_admm_ode}, respectively.  
In this section, we give the required background material.

For generality, consider the first-order dynamical system
\begin{equation}
\label{eq:dynamical}
\dot{Y} = F(Y,t) \qquad \mbox{with $Y(t_0) = Y_0$,}
\end{equation}
where $F: \mathbb{R}^p \times \mathbb{R}\to \mathbb{R}^p$, 
$Y=Y(t) \in \mathbb{R}^p$, and $Y_0\in\Re^{p}$.
When $F$ is Lipschitz continuous 
this initial value problem 
is well-posed.
Indeed, let $\Omega \subseteq \mathbb{R}^p\times\mathbb{R}$ and
suppose that $F$ is continuously differentiable on $\Omega$.
Let $(Y_0, t_0) \in
\Omega$. The Cauchy-Lipschitz theorem assures that
\eqref{eq:dynamical} has a unique solution $Y(t)$ on an open interval around
$t_0$ such that $Y(t_0) = Y_0$. This solution may be extended 
throughout $\Omega$. Moreover, the solution
is a continuous function of the initial condition $(Y_0, t_0)$, and
if $F$ depends continuously on some set of parameters, then it is 
also a continuous function of those parameters \cite{Smale}.

For a dynamical system in the form \eqref{eq:dynamical} we have
the following three basic types of stability.

\begin{definition}[Stability \cite{Smale}]
\label{th:stability}
A point $Y^\star$ such that $F(Y^\star, t) = 0$ for 
all $t \ge t_0$ is called
a critical point of 
the dynamical system \eqref{eq:dynamical}. We say the following:
\begin{itemize}
\item[(i)] $Y^\star$ is \emph{stable} if for every 
neighborhood $\mathcal{O}\subseteq\mathbb{R}^p$
of $Y^\star$, 
there exists a neighborhood $\bar{\mathcal{O}} \subseteq \mathcal{O}$ 
of $Y^\star$ such that every solution $Y(t)$ 
with initial condition
$Y(t_0)=Y_0 \in \bar{\mathcal{O}}$ is defined and remains 
in $\mathcal{O}$ for all $t > t_0$;
\item[(ii)] $Y^\star$ is \emph{asymptotically stable} if it
is stable and, additionally, satisfies 
$\lim_{t\to\infty} Y(t) = Y^\star$ for all $Y_0\in\bar{\mathcal{O}}$;
\item[(iii)] $Y^\star$ is \emph{unstable} if it is not stable.
\end{itemize}
\end{definition}

Stability implies the existence of a 
region around $Y^\star$, i.e., the basin of attraction, in 
which solutions to the differential equation remain in such a region 
provided the initial condition $Y_0$ 
is sufficiently close to $Y^\star$.
Asymptotic stability is stronger, further requiring 
that trajectories converge
to $Y^\star$. We note that 
convergence of the trajectory alone does not imply
stability.

Lyapunov formulated a strategy that enables one 
to conclude stability without integrating the equations of motion.

%\begin{theorem}[Lyapunov \cite{Smale,LaSalle}]
%\label{th:Lyapunov}
%Let $Y^\star$ be a critical point of \eqref{eq:dynamical}.  Also, let 
%$\Omega := \mathcal{O} \times [t_0,\infty)$ with 
%$\mathcal{O}\subseteq\mathbb{R}^n$ an open set containing $Y^\star$, 
%%$t_0$,
%and 
%$\mathcal{E}: \Omega \to \mathbb{R}$  a continuously differentiable function. 
%We thus have the following:
%\begin{enumerate}
%\item[(i)] if there exists a function $W(\cdot)$ that is 
%positive definite at $Y^\star$ over $\mathcal{O}$ such that 
%\begin{align}
%&\mathcal{E}(Y^\star,t) = 0 \ \ \text{for all $t \ge t_0$}, \label{L.1}\\
%&\mathcal{E}(Y,t) \ge W(Y) \ \ 
%\text{for all $t \ge t_0$ and $Y \in \mathcal{O}$,} \label{L.2}\\
%&\dot{\mathcal{E}}(Y, t) \leq 0 \ \ 
%\text{for all $t \ge t_0$ and $Y \in \mathcal{O}$,}\label{L.3}
%\end{align}
%then $Y^\star$ is stable and $\mathcal{E}$ is called
%a Lyapunov function;
%\item[(ii)] if condition~\eqref{L.3} is replaced by
%\begin{equation} \label{L.4}
%\dot{\mathcal{E}}(Y,t)
%\leq -W(Y) \ \ \text{for all $t\ge t_0$ and $Y \in \mathcal{O}$,}
%\end{equation}
%where $W(\cdot)$ is another positive definite function at $Y^\star$ over
%$\mathcal{O}$,
%then $Y^\star$ is asymptotically stable and $\mathcal{E}$ is called
%a strict Lyapunov function.  
%\end{enumerate}
%\end{theorem}

\begin{theorem}[Lyapunov \cite{Smale}]
\label{th:Lyapunov.2}
Let $Y^\star$ be a critical point of the dynamical system \eqref{eq:dynamical}.
Also, let $\mathcal{O}\subseteq\mathbb{R}^n$ be an open 
set containing $Y^\star$ and
$\mathcal{E}: \mathcal{O} \to \mathbb{R}$ be a continuously differentiable
function. We have the following:
\begin{enumerate}
\item[(i)] if $\mathcal{E}(\cdot)$ satisfies
\begin{align}
&\mathcal{E}(Y^\star) = 0, \label{L.1.simple}\\
&\mathcal{E}(Y) > 0 \ \ 
\text{for all $Y \in \mathcal{O}\setminus Y^\star$,} \label{L.2.simple}\\
&\dot{\mathcal{E}}(Y) \leq 0 \ \ 
\text{for all $Y \in \mathcal{O}\setminus Y^\star$,}\label{L.3.simple}
\end{align}
then $Y^\star$ is stable and $\mathcal{E}$ is called
a Lyapunov function;
\item[(ii)] if instead of \eqref{L.3.simple} we have the strict inequality
\begin{equation}\label{L.4.simple}
\dot{\mathcal{E}}(Y) < 0 \ \ 
\text{for all $Y \in \mathcal{O}\setminus Y^\star$,}
\end{equation}
then $Y^\star$ is asymptotically stable and $\mathcal{E}$ is called
a strict Lyapunov function.  
\end{enumerate}
\end{theorem}

The drawback of Lyapunov's approach is that it requires 
knowing an appropriate $\mathcal{E}(\cdot)$; unfortunately, 
there is no systematic procedure for constructing such a function. 
Also, note that Lyapunov's criteria are sufficient but not necessary.

\section{Stability and Convergence Analysis}
\label{sec:damped_second_order}

In this section we analyze the stability properties and rates of 
convergence of the dynamical systems associated 
with both ADMM and A-ADMM.

\subsection{Analysis of the Dynamical System for ADMM} \label{sec:dyn.admm}

\subsubsection*{Asymptotic stability}

The asymptotic stability of the dynamical 
system~\eqref{eq:admm_ode} associated with ADMM 
follows from Theorem~\ref{th:Lyapunov.2} with an 
appropriately chosen Lyapunov function.
\begin{theorem}
\label{thm:as.admm}
Let $X^\star$ be a strict local minimizer and an isolated 
stationary point of $V(\cdot)$, i.e., 
there exists $\mathcal{O} \subseteq \mathbb{R}^n$ such 
that $X^\star\in \mathcal{O}$, $\nabla V(X) \neq 0$ 
for all $X\in\mathcal{O}\setminus X^\star$, and 
\begin{equation}\label{strict.and.isolated}
V(X) > V(X^\star) \ \ \text{for all $X\in\mathcal{O}\setminus X^\star$.} 
\end{equation}
Then, it follows that $X^\star$ is an asymptotically stable 
critical point of the ADMM flow~\eqref{eq:admm_ode}.
\end{theorem}
\begin{proof}
Since $X^\star$ is a minimizer of $V(\cdot)$, 
it follows from first-order optimality conditions 
that $\nabla V(X^\star) = 0$. Combining this fact 
with Definition~\ref{th:stability} shows that $X^\star$ 
is a critical point of the dynamical system~\eqref{eq:admm_ode}.
To prove that $X^\star$ is asymptotically stable, let us define
\begin{equation}
\mathcal{E}(X) \equiv V(X) - V(X^\star)
\end{equation}
and observe from~\eqref{strict.and.isolated} that~\eqref{L.1.simple} 
and~\eqref{L.2.simple} hold.
Then, taking the total time derivative of $\mathcal{E}$ and using \eqref{eq:admm_ode} we have
\begin{equation} \label{dot.der.E}
\dot{\mathcal{E}}(X) = \big\langle \nabla V(X), \dot{X} \big\rangle = 
- \| A \dot{X} \|^2.
\end{equation}
Since $X^\star$ is assumed to be an isolated critical point, we know that
if $X\in\mathcal{O}\setminus X^\star$, then $\nabla V(X) \neq 0$, which in
light of~\eqref{eq:admm_ode} and Assumption~\ref{ass.basic} means that 
$\dot X \neq 0$.  Combining this conclusion with~\eqref{dot.der.E} and
Assumption~\ref{ass.basic} shows that if $X\in\mathcal{O}\setminus X^\star$,
then $\dot{\mathcal{E}}(X) < 0$, i.e., \eqref{L.4.simple} holds. 
Therefore, it follows from Theorem~\ref{th:Lyapunov.2} 
that $X^\star$ is an asymptotically stable
critical point of the dynamical 
system~\eqref{eq:admm_ode}.
\end{proof}

Some remarks concerning Theorem~\ref{thm:as.admm} are appropriate.
\begin{itemize}
\item If $V(\cdot)$ is strongly convex, then it has a 
unique minimizer.  Moreover, that unique minimizer will satisfy 
the assumptions of Theorem~\ref{thm:as.admm} with 
$\mathcal{O} = \Re^n$. Strong convexity of $V(\cdot)$ holds, 
for example, when either $f$ or $g$ is convex and the other 
is strongly convex (recall that $A$ has full column rank by assumption). 
Similar remarks also hold when $V(\cdot)$ is merely strictly convex.
\item If $X^\star$ satisfies the second-order sufficient optimality 
conditions for minimizing $V(\cdot)$, i.e., $\nabla V(X^\star) = 0$ 
and $\nabla^2 V(X^\star)$ is positive definite, then the 
assumptions of Theorem~\ref{thm:as.admm} will hold at $
X^\star$ for all sufficiently small neighborhoods 
$\mathcal{O}$ of $X^\star$.  
Note that in this case, the function $V(\cdot)$ need not be convex.
%\item Functions satisfying the Polyak-\L{}ojasiewicz (PL) condition~\cite{KariNutiSchm16} form a broader class of functions than the class of strongly convex functions. A consequence of satisfying the PL condition is that all stationary points are global minimizers (it does not imply a unique global minimizer).
\item It follows from~\eqref{dot.der.E} that $\dot{\mathcal{E}}(X) \leq 0$ 
for all $X$.  Thus, $X^\star$ will be stable (not necessarily 
asymptotically stable) without having to assume that $X^\star$ 
is an isolated stationary point of $V(\cdot)$.
\end{itemize}

\subsubsection*{Convergence rate}
For the dynamical system governing ADMM we are able to 
establish a convergence rate for how fast the objective 
function converges to its optimal value. 

\begin{theorem}
\label{rate.admm}
Let $X(t)$ be a trajectory of the ADMM flow~\eqref{eq:admm_ode},
with initial condition $X(t_0)=x_0$. Assume that $\argmin V \ne \emptyset$ and
let $V^\star \equiv \min_x V(x)$. 
Then, 
there is a constant $C > 0$ such that
\begin{equation}
V(X(t)) - V^\star \le \dfrac{C}{t} .
\end{equation}
\end{theorem}
\begin{proof}
Let $X^\star \in \argmin V$, thus $V(X^\star)=V^\star$.   
Consider
\begin{equation}
\label{eq:lyap_admm1}
\mathcal{E}(X,t) \equiv t \left[ V(X) - V(X^\star)  \right] + 
\tfrac{1}{2}\big\| A\left(X - X^\star\right)\big\|^2.
\end{equation}
By taking the total time derivative of $\mathcal{E}$, using
\eqref{eq:admm_ode}, and then 
the convexity of $V(\cdot)$, we find that
\begin{equation}
\begin{split}
\dot{\mathcal{E}} &= 
t \big\langle \nabla V(X), \dot{X} \big\rangle
+ V(X) - V(X^\star) +
\big\langle X - X^\star, A^T A \, \dot{X} \big\rangle \\
&= - t \big\| A\dot{X} \big\|^2 + V(X) - V(X^\star) + 
\big\langle X^\star - X, \nabla V(X) \big\rangle  \\
&\leq 0,
\end{split}
\end{equation}
from which we may conclude that 
$\mathcal{E}(X,t) \le \mathcal{E}(X_0,t_0)$ for all $t \geq t_0$.  
Combining this with the definition of $\mathcal{E}$ gives 
\begin{equation}
\begin{split}
\label{eq:rate_admm}
V(X) - V(X^\star)
&= \frac{1}{t}\mathcal{E}(X,t) - \frac{1}{2t} 
\| A\left(X - X^\star\right)\|^2  \\
&\leq \frac{\mathcal{E}(X_0,t_0)}{t},
\end{split}
\end{equation}
where we note that $\mathcal{E}(X_0,t_0) \geq 0$ since $V(\cdot)$ is convex.
\end{proof}

Some remarks concerning Theorem~\ref{rate.admm} are warranted.
\begin{itemize}
\item 
Theorem~\ref{rate.admm} holds under the 
assumption that $f$ and $g$ are convex.  
This is a strength compared with Theorem~\ref{thm:as.admm}, 
which has to make relatively strong assumptions about 
the critical point. Under those stronger assumptions, 
however, Theorem~\ref{thm:as.admm} gives a convergence 
result for the state $X$, whereas Theorem~\ref{rate.admm} 
only guarantees convergence of the objective
value.  
\item 
The $O(1/t)$ rate promised by Theorem~\ref{rate.admm} for 
the dynamical system~\eqref{eq:admm_ode} associated with ADMM 
agrees with the 
rate $O(1 / k)$ of ADMM when $V(\cdot)$ is  
assumed to be convex~\cite{Eckstein:2015,He:2012}.
\end{itemize}

\subsection{Analysis of the Dynamical System for A-ADMM}

\subsubsection*{Stability}
A stability result for the dynamical system~\eqref{eq:fast_admm_ode} 
associated with A-ADMM can be established by combining 
Theorem~\ref{th:Lyapunov.2} with an appropriately chosen Lyapunov function.

Let $Y_1=X$ and $Y_2 = \dot{X}$, and denote $Y=(Y_1,Y_2)$. Thus,
we are able to 
write the second-order dynamical system~\eqref{eq:fast_admm_ode} as 
the following system of first-order differential equations:
\begin{equation} 
\label{sys.equivalent}
\dfrac{d}{dt} \begin{pmatrix} Y_1 \\ Y_2 \end{pmatrix} = 
\begin{pmatrix} Y_2 \\ -\tfrac{r}{t} Y_2 - 
(A^T A)^{-1} \nabla V(Y_1) \end{pmatrix}.
\end{equation}
We can now give conditions on minimizers $X^\star$ of $V(\cdot)$ 
that ensure that $Y^\star = (X^\star,0)$ is a stable 
critical point for~\eqref{sys.equivalent}.

\begin{theorem} 
\label{thm:stabile.aadmm}
If $X^\star$ be a strict local minimizer
of $V(\cdot)$, i.e., there 
exists $\mathcal{O} \subseteq \mathbb{R}^n$ such 
that $X^\star\in \mathcal{O}$
and 
\begin{equation}\label{strict.and.isolated.2}
V(X) > V(X^\star) \ \ \text{for all $X\in\mathcal{O}\setminus X^\star$,} 
\end{equation}
then $Y^\star = (X^\star,0)$ 
is a stable critical point for the 
dynamical system~\eqref{sys.equivalent}, which is equivalent to 
the A-ADMM flow~\eqref{eq:fast_admm_ode}.
\end{theorem}
\begin{proof}
Since $X^\star$ is a minimizer of $V(\cdot)$, it follows from 
first-order optimality conditions that $\nabla V(X^\star) = 0$. 
Combining this with Definition~\ref{th:stability} 
shows that $Y^\star = (X^\star,0)$ is a critical 
point of the first-order dynamical system~\eqref{sys.equivalent}.

Next, we prove that $Y^\star = (X^\star,0)$ is stable.
Let $\mathcal{O}\subseteq \Re^n$ 
and define $\mathcal{E}:\mathcal{O}  
\to \Re$ as
\begin{equation}
\label{eq:energy1}
\mathcal{E}(Y) = \tfrac{1}{2} \| A Y_2 \|^2 + V(Y_1) - V(X^\star).
\end{equation}
Note that 
$\mathcal{E}(Y^\star) = 0$, 
i.e., condition \eqref{L.1.simple} holds.  Also,
since $X^\star$ is isolated,
$\mathcal{E}(Y) > 0$ for all $Y \ne Y^\star$, so that \eqref{L.2.simple}
holds.
If we take the total time derivative of \eqref{eq:energy1} we obtain
\begin{equation}
\begin{split}
\dot{\mathcal{E}}
&= \big\langle \nabla_{Y_1} \mathcal{E}, \dot{Y}_1 \big\rangle
+\big\langle \nabla_{Y_2} \mathcal{E}, {\dot{Y}_2} \big\rangle 
 \\
&= \big\langle \nabla V(Y_1), Y_2 \big\rangle
-\left\langle A^T A Y_2, \dfrac{r}{t} Y_2 + (A^T A)^{-1} 
\nabla V(Y_1) \right\rangle  
 \\ 
&= -\tfrac{r}{t} \| A Y_2 \|^2  .
\end{split}
\end{equation}
Thus, $\dot{\mathcal{E}}(Y) \le 0$ for all $Y \in{\mathcal{O}}$, 
i.e., \eqref{L.3.simple} 
holds.  
This implies that
$Y^\star = (X^\star,0)$ is stable, as claimed.
\end{proof}

We remark that the discussions in the first two bullet 
points of the subsection ``Asymptotic stability" 
in Section~\ref{sec:dyn.admm} also apply to 
Theorem~\ref{thm:stabile.aadmm}.  We do not repeat them for brevity. 
We also note that the stability of system 
\eqref{eq:nesterov_ode} was not 
considered by \citet{Candes:2016,Wibisono:2016,Krichene:2015,Wilson:2016}.
In contrast,  Theorems \ref{thm:as.admm} and \ref{thm:stabile.aadmm} provide
a simple argument for the stability
of \eqref{eq:admm_ode}~and~\eqref{eq:fast_admm_ode}, respectively, based only on
Theorem~\ref{th:Lyapunov.2}. However, contrary to the first-order 
ADMM flow~\eqref{eq:admm_ode},
it is not obvious how to apply Theorem~\ref{th:Lyapunov.2} to
the second-order A-ADMM flow~\eqref{eq:fast_admm_ode}
to obtain asymptotic stability 
without further assumptions on the critical point.
However, we give a case where 
asymptotic stability holds
in the end of this section.

Let us mention existing results regarding
the convergence of trajectories of the system \eqref{eq:nesterov_ode}
when $X(t)$ is an 
element of a Hilbert space.
Convergence of trajectories for convex and even some particular
cases of non-convex $f(\cdot)$ was studied in \citet{Cabot:2009}.
If $r > 3$
and $\argmin f \ne \emptyset$,
then the trajectory of the system weakly converges to 
some minimizer of $f(\cdot)$, even in the presence of small 
perturbations \cite{Attouch:2016,May:2016}.
These results should extend naturally to \eqref{eq:fast_admm_ode}, but
we avoided diving in this direction
since it would deviate from our main goal. 
It is important to note, however, that convergence of the trajectories
do not necessarily imply stability.

\subsubsection*{Convergence rate}
We now consider the convergence rate of the dynamical system 
\eqref{eq:fast_admm_ode}.% associated to A-ADMM.

\begin{theorem}
\label{th:rate_damped_gradient}
Let $X(t)$ be a trajectory of the A-ADMM flow~\eqref{eq:fast_admm_ode},
with initial conditions $X(t_0)=x_0$ and $\dot{X}(t_0)=0$.
Assume that $\argmin V \ne \emptyset$ and let $V^\star \equiv \min_x V(x)$.
If $r\ge 3$, then there is some constant $C > 0$ such that
\begin{equation}
\label{eq:rate_damped_gradient}
V(X(t)) - V^\star \le \dfrac{C}{t^2}.
\end{equation}
\end{theorem}
\begin{proof}
Following \citet{Candes:2016,Wibisono:2016}, we define
$\eta:[t_0,\infty) \to \Re$ as $\eta(t) = 2\log\left(\tfrac{ t}{r-1}\right)$ 
and
\begin{equation}
\mathcal{E}(Y,t) = e^{\eta}
\big[ V(Y_1) - V(X^\star) \big]  
+\tfrac{1}{2}\big\| A\big( Y_1 - X^\star + 
e^{\eta/2} Y_2 \big) \big\|^2
\end{equation}
where $X^\star$ is any minimizer of $V(\cdot)$, and thus 
$V(X^\star) = V^\star$.
Note that $\mathcal{E} \ge 0$ and its total time derivative is given by
\begin{equation}
\begin{split}
\dot{\mathcal{E}} &= 
\langle \nabla_{Y_1} \mathcal{E}, \dot{Y}_1 \rangle
+\langle \nabla_{Y_2} \mathcal{E}, \dot{Y}_2 \rangle
+ \partial_t \mathcal{E} \\
&= \left\langle \nabla_{Y_1} \mathcal{E} - 
\tfrac{r}{t}\nabla_{Y_2} \mathcal{E}, Y_2 \right\rangle
-\langle \nabla_{Y_2} \mathcal{E}, (A^TA)^{-1}\nabla_{Y_1} V \rangle
+ \partial_t \mathcal{E}
\end{split}
\end{equation}
where we made use of \eqref{sys.equivalent}.
Observe that
\begin{align}
\nabla_{Y_1} \mathcal{E} &= e^\eta \nabla_{Y_1} V + 
A^TA \big( Y_1 - X^\star + e^{\eta/2}Y_2 \big), \\
\nabla_{Y_2} \mathcal{E} &= e^{\eta/2} 
A^TA \big( Y_1 - X^\star + e^{\eta/2}Y_2  \big), \\
\partial_t \mathcal{E} &= \dot{\eta} e^{\eta}\big(V(Y_1) - V(X^\star) \big)  
 + \tfrac{\dot\eta}{2} e^{\eta/2} \big\langle Y_1 - X^\star + 
e^{\eta/2} Y_2, A^T A \, Y_2
\big\rangle,
\end{align}
and also that
\begin{equation}\label{key.id}
%\mathcal{E} \ge 0
%\ \ \text{and} \ \ 
e^{-\eta/2} + \tfrac{1}{2}{\dot\eta}
= \tfrac{r}{t}.
\end{equation}
Therefore, using the convexity of $V(\cdot)$ we obtain
\begin{equation}
\begin{split}
\dot{\mathcal{E}} &= \dot{\eta}e^{\eta} \big(V(Y_1) - V(X^\star)\big)
- e^{\eta / 2} \langle Y_1 - X^\star, \nabla V \rangle \\
&\le - e^{\eta/2}(1 - \dot{\eta} e^{\eta/2}) 
\left(V(Y_1) - V(X^\star)\right) \\
&= - \dfrac{t(r-3)}{(r-1)^2} \big(V(Y_1) - V(X^\star)\big) ,
\end{split}
\end{equation}
so that $\dot{\mathcal{E}}\le 0$, implying
$\mathcal{E}(Y,t) \le \mathcal{E}((X(t_0), \dot{X}(t_0)), t_0)$.
By the definition of $\mathcal{E}(\cdot)$ we thus have
\begin{equation}
\label{eq:rate_damped_gradient2}
\begin{split}
V(Y_1) - V(X^\star) &\le e^{-\eta}
\mathcal{E}(Y,t) \\
&\le e^{-\eta} \mathcal{E}(X(t_0), \dot{X}(t_0), t_0).
\end{split}
\end{equation}
To conclude the proof, observe that $e^{\eta} = {t}^2/(r-1)^2$.
\end{proof}

Theorem~\ref{th:rate_damped_gradient}
suggests that
A-ADMM has a convergence rate of $O\big(1/k^2\big)$ 
for convex functions. This agrees with 
the result by \citet{Goldstein:2014}, which
assumes strong convexity of both $f$ and $g$, and also that
$g$ is quadratic; see \eqref{eq:minimize}.
Moreover, \citet{Goldstein:2014}
do not bound
the objective function as in \eqref{eq:rate_damped_gradient} 
but the combined residuals.
A convergence rate of $O(1/k^2)$ was also obtained for an accelerated
variant of the Douglas-Rachford splitting method for problem 
$\min_x f(x) + g(x) $ when both $f$ and $g$ are convex and $f$ is
quadratic \cite{Patrinos:2014}.
To the best of our knowledge, there is no
$\mathcal{O}(1/k^2)$ convergence proof for A-ADMM assuming only
convexity. It would be interesting to consider the convergence rate
of A-ADMM directly through a discrete analog of the 
Lyapunov function used in the above theorem, in the same spirit as
\citet{Candes:2016} considered for A-GD
and more recently \citet{Attouch:2016} considered for a perturbed version
of A-GD.

\subsubsection*{Asymptotic stability}
Under stronger conditions than in Theorem~\ref{thm:as.admm},
we have asymptotic stability of the dynamical system
\eqref{eq:fast_admm_ode} associated with A-ADMM.

\begin{theorem}
\label{th:asymptotic_fast_admm}
Let $X^\star \in \mathcal{O}$, for some 
$\mathcal{O}\subseteq \mathbb{R}^n$, 
be a local minimizer
of $V(\cdot)$ satisfying 
\begin{equation}
\label{eq:forcing}
V(X) - V(X^\star) \ge \phi(\| X - X^\star \|) \qquad
\mbox{for all $X \in \mathcal{O}$},
\end{equation}
where $\phi:[0, \infty)\to[0, \infty)$ is a forcing 
function \cite{Ortega} such that
for any %sequence 
$\{ \xi_k \} \subset [0, \infty)$,
$\lim_{k\to\infty}\phi(\xi_k) = 0$ implies 
$\lim_{k\to\infty} \xi_k = 0$.
Moreover, suppose that the conditions of Theorem~\ref{th:rate_damped_gradient}
hold over $\mathcal{O}$.
Then, it follows that $Y^\star = (X^\star, 0)$ is an asymptotically stable
critical point of the dynamical system
\eqref{sys.equivalent}, 
which is equivalent to the A-ADMM flow~\eqref{eq:fast_admm_ode}.
\end{theorem}
\begin{proof}
Consider \eqref{eq:forcing} over a trajectory $X = X(t)$. Using 
\eqref{eq:rate_damped_gradient} and~\eqref{eq:forcing} we have
$\lim_{t\to\infty} \phi(\| X(t) - X^\star\|) = 0$, 
which combined with the properties of the forcing function gives
\begin{equation}
\label{eq:conv_traj}
\lim_{t\to\infty}\| X(t) - X^\star\| = 0.
\end{equation}
Denote $Y(t)= (Y_1(t), Y_2(t)) = (X(t), \dot{X}(t))$. From the proof
of Theorem~\ref{th:rate_damped_gradient}, i.e., the definition of
$\mathcal{E}(\cdot)$ and $\dot{\mathcal{E}} \le 0$, we also have that
$\| A(Y_1 - X^\star) + e^{\eta/2} Y_2 \|^2 \le C$, where
$C = \mathcal{E}(X(t_0), \dot{X}(t_0))$, hence
\begin{equation}
e^{\eta(t)} \| Y_2(t)\|^2 \le C + \| A(Y_1(t) - X^\star) \|^2 .
\end{equation}
This implies that $\lim_{t\to\infty} \| Y_2(t) \| = 0$
upon using \eqref{eq:conv_traj}. 

We showed that $\lim_{t\to\infty}Y(t) = (X^\star,0)$.
From Theorem~\ref{thm:as.admm} we already know that $Y^\star  =
(X^\star, 0)$ is a stable critical point of \eqref{sys.equivalent}.
From these two facts, and Definition~\ref{th:stability}, we thus
conclude that $Y^\star$ is asymptotically stable, as claimed.
\end{proof}

Condition \eqref{eq:forcing} holds,
for instance,
for both uniformly convex functions and strongly convex functions.

\section{A Numerical Example}

We numerically verify that the differential 
equations \eqref{eq:admm_ode}
and \eqref{eq:fast_admm_ode} accurately model ADMM and A-ADMM,
respectively, when $\rho$ is large as needed to derive the
continuous limit. The numerical integration of the first-order system
\eqref{eq:admm_ode} is straightforward; we use a 4th order Runge-Kutta
method (an explicit Euler method could also be employed). 
The numerical integration of \eqref{eq:fast_admm_ode}
is more challenging due to strong oscillations. To obtain a 
faithful discretization
of the continuous dynamical system \eqref{eq:fast_admm_ode}, i.e.,
one that preserves its properties, 
a standard approach is to use a Hamiltonian symplectic integrator, which
is designed to preserve the phase-space volume.
Consider the Hamiltonian
\begin{equation}
\label{eq:ham}
\mathcal{H} \equiv \tfrac{1}{2}e^{-\xi(t)} \big\langle P, (A^T A)^{-1} P
\big\rangle + e^{\xi(t)} V(X), 
\end{equation}
where $\xi(t) \equiv r \log t$ and $P = e^{\xi} (A^T A) \dot{X}$ 
is the canonical
momentum. Hamilton's equations are given by
\begin{equation}
\label{eq:ham_eq}
\dot{X} = \nabla_P \mathcal{H}
\qquad \text{and}  \qquad
\dot{P} = -\nabla_X \mathcal{H}.
\end{equation}
One can check that \eqref{eq:ham_eq} together with \eqref{eq:ham} 
is equivalent
to \eqref{eq:fast_admm_ode}. The simplest scheme is the symplectic Euler
method, %which for this case is 
which for equations \eqref{eq:ham_eq} with \eqref{eq:ham} is 
given explicitly as
\begin{subequations}
\label{eq:symp_euler}
\begin{align}
p_{k+1} &= p_k - h e^{\xi(t_k)} \nabla V(x_k), \\
x_{k+1} &= x_k + h e^{-\xi(t_k)} \big(A^T A\big)^{-1} p_{k+1}, \\
t_{k+1} &= t_k + h,
\end{align}
\end{subequations}
where $h > 0$ is the step size. Thus, we compare the iterates
\eqref{eq:symp_euler} with the A-ADMM algorithm. A simple
example is provided in Figure~\ref{fig:numerical}, which illustrate
our theoretical results.
%We can see that
%numerical solutions of the continuous dynamical systems are close
%to the algorithms.
%, and that Theorem~\ref{th:rate_admm} and
%Theorem~\ref{th:rate_damped_gradient} hold.

\begin{figure}
\centering
\includegraphics[scale=.7]{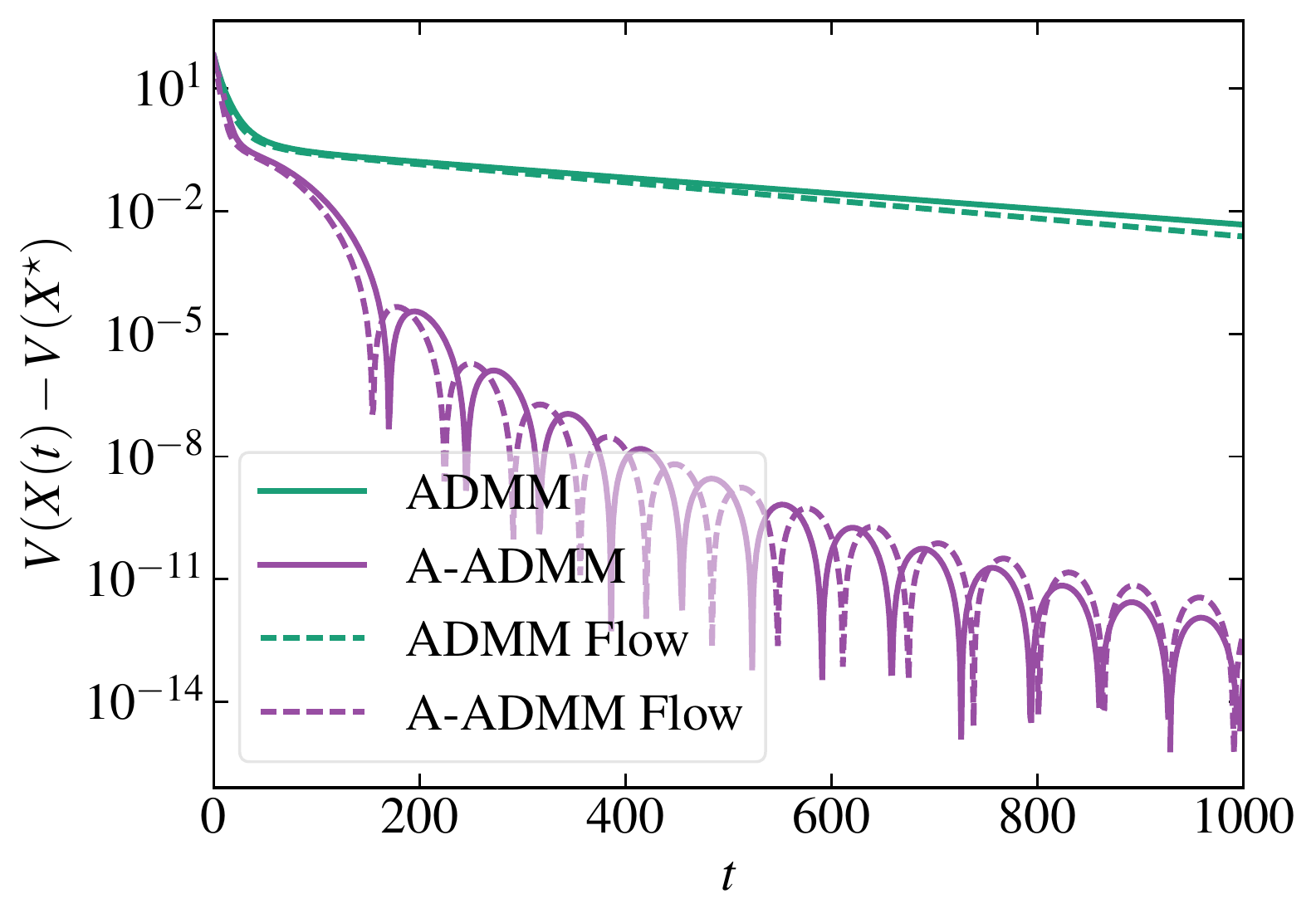}
\caption{
\label{fig:numerical}
$\min_x V(x)$ such that $z=A x$ with 
$V(x) = \tfrac{1}{2}\langle x, M x\rangle$,
where $M \in \mathbb{R}^{60\times 60}$ is a random matrix with
$40$ zero eigenvalues and the others are uniformly distributed
on $[0,10]$, and $A$ is a full column rank random matrix with condition
number $100$.
We solve this using ADMM versus solutions to \eqref{eq:admm_ode} through
4th order Runge-Kutta, and A-ADMM
versus solutions to \eqref{eq:fast_admm_ode} through the symplectic Euler
method \eqref{eq:symp_euler}. We choose $r = 10$ and $\rho = 50$. 
The initial conditions are $X(0) = x_0 = 5 (1,1,\dotsc,1)^T$ 
and $\dot{X}(0)=0$.
The curves
are close and the rates \eqref{eq:rate_admm} and
\eqref{eq:rate_damped_gradient} hold.
}
\end{figure}

\section{Conclusions}
\label{sec:conclusions}

Previous work considered dynamical systems for continuous limits of gradient-based methods
for unconstrained optimization~\cite{Candes:2016,Wibisono:2016,Krichene:2015}. Our paper builds upon these
results by showing
that the continuous limits of ADMM and A-ADMM correspond to 
first- and second-order 
dynamical systems, respectively; 
see Theorems~\ref{th:admm_ode}~and~\ref{th:fast_admm_ode}.  
Next, 
using a Lyapunov stability analysis, 
we presented conditions %on the objective function 
that ensure stability and asymptotic
stability of the dynamical systems;
see Theorems~\ref{thm:as.admm}, \ref{thm:stabile.aadmm} 
and \ref{th:asymptotic_fast_admm}.
Furthermore, 
in Theorem~\ref{rate.admm} 
we obtained a convergence rate of $O(1/t)$ for the
dynamical system related to ADMM,
which is consistent with the known
$O(1/k)$ convergence rate of the discrete-time ADMM, %and 
whereas in 
Theorem~\ref{th:rate_damped_gradient} we obtained
a convergence rate of $O(1/t^2)$ for the dynamical system related to %for 
A-ADMM, which is a new result since this rate
is unknown for discrete-time A-ADMM.
We also showed that the dynamical system associated to \mbox{A-ADMM}
is a Hamiltonian system, 
%equations \eqref{eq:ham} and \eqref{eq:ham_eq},
and by employing a simple symplectic integrator verified numerically
the agreement between discrete- and continuous-time dynamics.

The results presented in this paper may be useful for understanding the behavior
of ADMM and A-ADMM for non-convex problems as well. For instance,
following ideas from \citet{Jin:2017} and \citet{Lee:2017} 
an analysis of the center manifold of the dynamical systems \eqref{eq:admm_ode}
and \eqref{eq:fast_admm_ode} can provide valuable insights on the stability
of saddle points, which is considered a major issue 
in non-convex optimization. Also,
ADMM is well-suited to large-scale problems in statistics and machine 
learning, being equivalent to Douglas-Rachford splitting and
closely related to other algorithms 
such as augmented Lagrangian methods, 
dual decomposition, 
and Dykstra's alternating projections.
Therefore, our results may give new
insights into these methods as well.

\begin{acknowledgments}
This work was supported by grants
ARO MURI W911NF-17-1-0304 and
NSF 1447822.  
\end{acknowledgments}

\bibliographystyle{utphys}
\bibliography{biblio.bib}

\end{document}